\documentclass[11pt]{amsart}
\usepackage{amsmath,amssymb, mathrsfs}
\usepackage{xypic}
\usepackage[colorlinks=true,linkcolor=black,citecolor=black,urlcolor=black]{hyperref}

\usepackage[french,english]{babel}

\usepackage[french,english]{babel}
\usepackage[T1]{fontenc}

 \usepackage{mathrsfs}
 \usepackage{color}
 \usepackage{url}

 \usepackage[foot]{amsaddr}

\newcommand{\rem}[1]{} % Handy for vanishing large parts of the file without erasing them

\newcommand{\N}{\mathbb{N}}

\newcommand{\bbA}{{\mathbb{A}}}

\newcommand{\calF}{{\mathcal{F}}}

\newcommand{\frakp}{{\mathfrak{p}}}
\newcommand{\frakq}{{\mathfrak{q}}}

\newcommand{\suchthat}{\,:\,}

\newcommand{\units}[1]{{#1}^\times}

\DeclareMathOperator{\Max}{Max} %
\newcommand{\op}{\mathrm{op}}

\DeclareMathOperator{\Spec}{Spec}

\newtheorem{thm}{Theorem}[section]
\newtheorem*{thm*}{Theorem}
\newtheorem{lem}[thm]{Lemma}
\newtheorem{prp}[thm]{Proposition}
\newtheorem{cor}[thm]{Corollary}

\newtheoremstyle{roman} % name
    {8.0pt plus 2.0pt minus 4.0pt}                    % Space above
    {8.0pt plus 2.0pt minus 4.0pt}                    % Space below
    {\normalfont}                % Body font
    {}                           % Indent amount
    {\bfseries}                  % Theorem head font
    {.}                          % Punctuation after theorem head
    {5pt plus 1pt minus 1pt}     % Space after theorem head
    {}  % Theorem head spec (can be left empty, meaning �normal�)

\theoremstyle{roman}

\newtheorem{remark}[thm]{Remark}

\theoremstyle{plain}

\newcommand{\nMat}[2]{\mathrm{M}_{#2}(#1)}
\newcommand{\Mat}{\operatorname{M}}

\numberwithin{equation}{section}

\author{Uriya A. First}

\address
{Department of Mathematics \\
University of British Columbia \\
Vancouver
\\
CANADA}

\email{uriya.first@gmail.com}
\author{Zinovy Reichstein}

\email{reichst@math.ubc.ca}
\thanks {The second author has been partially
supported by an NSERC Discovery Grant.}

\begin{document}

\title[On the number of generators of an algebra]{On the number
of generators of an algebra}

\begin{abstract}
A classical theorem of Forster asserts that
a finite module $M$ of rank $\leq n$ over a Noetherian ring
of Krull dimension $d$
can be generated by $n + d$ elements. We prove a generalization
of this result, with ``module'' replaced by ``algebra''.
Here we allow arbitrary finite algebras,
not necessarily unital, commutative or associative. Forster's theorem
can be recovered as a special case by viewing a module as an algebra
where the product of any two elements is $0$.
\end{abstract}

% addresses altered using the package amsaddr.

\maketitle

\rem{
\maketitle

\vspace{-.5in}

\begin{otherlanguage}{french}
\begin{abstract}
Un th\'eor\`eme classique de Forster affirme que tout module $M$
de type fini et de {rang~$\leq n$} sur un anneau noeth\'erien
de dimension de Krull $d$ peut \^etre
engendr\'e par $n+d$ \'el\'ements.
% g\'en\'er\'e par $n+d$ \'el\'ements.
Nous prouvons une g\'en\'eralisation de ce r\'esultat o\`u le mot ``module''
est remplac\'e par ``alg\`ebre''. Les alg\`ebres consid\'er\'ees ici sont
de type fini mais non n\'ecessairement unitaires, commutatives ou m\^eme
associatives. Le th\'eor\`eme de Forster peut \^etre d\'eduit du cas
particulier o\`u un module est vu comme une alg\`ebre dont le produit
de deux \'el\'ements quelconques est $0$.
\end{abstract}
\end{otherlanguage}
}

\section{Introduction}

Throughout this paper $R$ will denote a commutative Noetherian ring with $1$.
For $\frakp\in \Spec R$, $R(\frakp)$ will denote the fraction field
of $R/\frakp$.  The starting point of this paper is the following
classical theorem of Forster.

\begin{thm}[{\cite{forster}}]
\label{thm.forster}
Suppose $R$ is Noetherian
of Krull dimension $d$ and let $M$ be a finite $R$-module.
If the $R(\frakp)$-module $M(\frakp) := M \otimes_R R(\frakp)$
can be generated by $n$ elements for every $\frakp\in\Max R$,
then $M$ can be generated by $n+d$ elements.
\end{thm}

Swan~\cite{swan} showed that Theorem~\ref{thm.forster} remains
valid when the Krull dimension of $R$ is replaced by the dimension
of $\Max R$.
\footnote{Recall that the dimension of a topological
space $X$ is the maximal length $d$ of
a chain $\emptyset\neq X_0 \subsetneq
X_1 \subsetneq \dots \subsetneq X_d \subseteq X$ of closed irreducible
subsets (or $- \infty$). The Krull dimension of $R$ is
the dimension of $\Spec R$ endowed with the Zariski topology.
The maximal spectrum $\Max R$ is a subspace of $\Spec R$, hence
$\dim \Max R  \leqslant \dim \Spec R $.}
Further generalizations and refinements of Forster's Theorem
can be found in~\cite{swan, eisenbud-evans, warfield, kks}.
This note offers yet another generalization, replacing
the finite $R$-module $M$
by a finite $R$-algebra $A$, i.e., by a finitely generated $R$-module
$A$ with an $R$-bilinear multiplication
map $A \times A \to A$. This bilinear map
can be arbitrary; we do not require $A$ to be
commutative or associative, or to have a unit element.
For $\frakp\in \Spec R$, write $A(\frakp): =A\otimes_RR(\frakp)$.
	
	\begin{thm}\label{thm.main}
    Assume $\dim\Max R=d$ and
let $A$ be a finite $R$-algebra
such that $A(\frakp)$ can be generated by $n$ elements as a non-unital
$R(\frakp)$-algebra for every $\frakp \in \Max R $.
Then $A$ can be generated by $n+d$ elements as a non-unital
$R$-algebra.
\end{thm}

In the case  where the multiplication map $A \times A \to A$
is identically zero, we recover Forster's Theorem~\ref{thm.forster}.
Other applications of Theorem~\ref{thm.main}
can be found in Section~\ref{sect.applications}.
Before proceeding with the proof, we remark that
our argument also proves the following variants of Theorem~\ref{thm.main}.

(i) If $A$ is a unital algebra,
Theorem~\ref{thm.main} remains valid if we replace ``generated
as a non-unital algebra'' with ``generated as a unital algebra''.

(ii) Both the original and the unital versions of Theorem~\ref{thm.main}
remain valid in the setting of~\cite{swan}, where
$A$ is equipped with a left $\Lambda$-module structure, $\Lambda$ being an $R$-algebra,
and generation means generation   as an $R$-algebra carrying
an additional $\Lambda$-module structure. Note that, unlike
\cite[Theorem~1]{swan}, we do not require $\Lambda$ to
be finitely generated as an $R$-module.

(iii) More generally, $A$ can be taken to be a finite $R$-multialgebra,
i.e., a finite right $R$-module equipped with an indexed family
of homogeneous maps $\{f_i:A^{n_i}\to A\}_{i\in I}$.
Here we say that $f:A^{k}\to A$ is
$(m_1,\dots,m_{k})$-homogeneous, if
$f(a_1r_1,\dots,a_{k}r_{k})= f(a_1,\dots,a_k)
r_1^{m_1} \dots r_k^{m_k}$
for all $a_1,\dots,a_{k}\in A$ and $r_1,\dots,r_{k}\in R$.
Note that $k = 0$ is allowed; in this case $f$ can
be any map from $A^0=0$ to $A$.
The family is $\{f_i\}_{i\in I}$ is clearly amenable
to base change, hence $A(\frakp)$ carries the structure of an
$R(\frakp)$-multialgebra. A multisubalgebra of $A$ is an $R$-submodule
closed under each $f_i$, and the multisubalgera generated
by $S \subset A$ is the smallest multisubalgbera containing $S$.
Multialgebras can be used to encode many types of structures.
For example, a (non-unital) $R$-algebra structure on $A$
is a $(1,1)$-homogeneous map $A^2\to A$,
a unit element can be specified by a  map $A^0 \to A$,
an involution by a $(1)$-homogeneous map $A\to A$, a quadratic Jordan
algebra structure by a $(2,1)$-homogeneous map $A^2\to A$, etc.
Furthermore, if $\Lambda$ is an associative $R$-algebra, then a left
$\Lambda$-module structure, as in (ii), can be represented by
the family of $(1)$-homogeneous maps $\{f_{\lambda} \colon A \to A\}_{\lambda\in\Lambda}$,
given by $f_\lambda(a)=\lambda a$.%$a \mapsto \lambda a$.

\section{Preliminary Lemmas}
\label{sect.prelim}
    Let $A$ be a finite $R$-algebra.
    For $\frakp\in\Spec R$ and $a\in A$,   denote the image of $a$
	in $A(\frakp)$ by $a(\frakp)$.
	
	\begin{lem}\label{LM:local-generation}
	Let $a_1,\dots,a_n\in A$. Then $a_1,\dots,a_n$ generate
	$A$ as an $R$-algebra if and only if for all $\frakp\in \Max R$,
	the elements $a_1(\frakp),\dots,a_n(\frakp)$ generate $A(\frakp)$
		as an $R(\frakp)$-algebra.
	\end{lem}
	
	\begin{proof}
		Let $B$ be the $R$-subalgebra generated by $a_1,\dots,a_n$.
		The map $B(\frakp)\to A(\frakp)$ induced by the inclusion $B\hookrightarrow A$
		is an isomorphism for all $\frakp$. Since
		$A$ is a finite $R$-algebra, Nakayama's Lemma implies that the map $B_\frakp\to A_\frakp$
		is an isomorphism  for all $\frakp\in\Max R$.
        It is well known that this implies $B=A$.
        %We claim that this
        %implies $B=A$.
		%Indeed, assume the contrary: $A/B\neq 0$. Choose some $0\neq x\in A/B$ and $\frakp\in \Max R$
		%such that $\ann_R(x)\subseteq \frakp$. Then the image of $x$ in $(A/B)_\frakp=A_\frakp/B_\frakp$
		%is nonzero, a contradiction. This
        %completes the proof of the claim and thus of the lemma.
	\end{proof}
	
	\begin{lem}\label{lem.open-generation}
Suppose $a_1,\dots,a_n\in A$ and $\frakp\in\Spec R$.
If $a_1(\frakp),\dots,a_n(\frakp)$ generate
$A(\frakp)$ as an $R(\frakp)$-algebra,
then there exists an open neighborhood $U$~of $\frakp$ in $\Spec R$
such that $a_1(\frakq),\dots,a_n(\frakq)$ generate $A(\frakq)$
for any $\frakq\in U$.
	\end{lem}
	
	\begin{proof} By our assumption, there exist (non-associative)
                monomials $\omega_1,\dots,\omega_t$ on $n$ letters
		such that
		$A(\frakp)$ is spanned
		by $\{\omega_i(a_1(\frakp),\dots,a_n(\frakp))\}_{i=1}^t$
                as an $R(\frakp)$-module.
		Write $b_i=\omega_i(a_1,\dots,a_n)$. By Nakayama's Lemma,
		$A_\frakp$ is spanned
		as an $R_\frakp$-module by the images of $b_1,\dots,b_t$.
		Let $B=\sum_ib_iR$. Then $(A/B)_\frakp=0$. Since $A$
		is finitely generated, there is $s\in R\setminus \frakp$
		such that $(A/B)s=0$. Thus, for any $\frakq\in\Spec R$
		not containing $s$, we have $(A/B)_\frakq=0$. Hence,
		$a_1(\frakq),\dots,a_n(\frakq)$ generate $A(\frakq)$
as an $R(\frakq)$-algebra.
	\end{proof}
	
	To state the next lemma, we need some additional notation.
	Let $n\in\N$. For any commutative associative unital $R$-algebra $S$,
	let  $A_S=A\otimes_RS$ and write
	\[
	V_n(S)=\left\{(a_1,\dots, a_n)\in A_S^n\suchthat \text{$a_1,\dots,a_n$ generate $A_S$
	as an $S$-algebra}\right\}\ .
	\]
	For all $0\leq i\leq n$, we further let
	\[
	V_{n,i}(S)=
    \left\{(a_1,\dots,a_i)\in A_S^i\suchthat \text{$\exists$ $a_{i+1},\dots,a_n\in A_S$
    such that $(a_1,\dots,a_n)\in V_n(S)$}\right\}\ .
	\]
	%denote the image of $V_n(S)$ in $A_S^i$ under the map $(a_1,\dots,a_n)\mapsto (a_1,\dots,a_i)$.

	\begin{lem}\label{LM:open-completable-generation}
		Let $\frakp\in\Spec R$, let $a_1,\dots,a_i\in A$,
		and assume that $(a_1(\frakp),\dots,a_i(\frakp))\in V_{n,i}(R(\frakp))$.
		Then there exists an open neighborhood $U$ of $\frakp$ in $\Spec R$
		such that $(a_1(\frakq),\dots,a_i(\frakq))\in V_{n,i}(R(\frakq))$ for all $\frakq\in U$.
	\end{lem}
	
	\begin{proof}
		There are $b_{i+1},\dots,b_n\in A(\frakp)$ such
that $a_1(\frakp),\dots,a_i(\frakp),b_{i+1},\dots,b_n$
		generate $A(\frakp)$. After multiplying
$b_{i+1},\dots,b_n$ by suitable invertible elements of $R(\frakp)$,
we may assume that each $b_j$ is the image of some element $a_j\in A$.
By Lemma~\ref{lem.open-generation}, there is an open neighborhood $U$ of
		$\frakp$ such that for all
		$\frakq\in U$, the elements $a_1(\frakq),\dots,a_n(\frakq)$ generate $A(\frakq)$.
		In particular, $(a_1(\frakq),\dots,a_i(\frakq))\in V_{n,i}(R(\frakq))$.
	\end{proof}

\section{Proof of Theorem~\ref{thm.main}}
\label{sect.proof}

We claim that
for every $0\leq j\leq n+d$, there exist
elements $a_1, \dots, a_{j} \in A$ and a partition of $X:=\Max R$
into locally closed subsets
$
X=\calF_0^{(j)}\sqcup \calF_1^{(j)} \sqcup \dots\sqcup\calF_n^{(j)}
$
with the following properties:
		\begin{enumerate}
\item[(1)] For any $i \geqslant 1$ and any $\frakp \in \calF_i^{(j)}$,
there are $t_1,\dots,t_i\in\{1,\dots,j\}$
such that \[ (a_{t_1}(\frakp),\dots,a_{t_i}(\frakp))\in V_{n,i}(R(\frakp)). \]
\item[(2)] %$i + \dim X\geq j+\dim \, \calF_i^{(j)}$
$\dim \calF_i^{(j)}\leq \dim X+i-j$
for
every $0 \leq i< n$. % and $0\leq j\leq n+d$.
		\end{enumerate}
% Informally speaking condition (2) says that if $\frakp \in \calF_i^{(j)}$,
% then $a_1, \dots, a_j$ can be completed to a generating set in
% of $A(\frakp)$ (as an $R(\frakp)$-algebra) by adding $n - i$ elements.
% As $j$ increases and $0 \leqslant i < n-1$, we want $\calF_{i}^{(j)}$ to
% descrease as fast as possible, so that by the time $j$ reaches $n + d$,
% they all become empty, and
% Assuming this holds

For $j=n+d$, condition (2) implies that $\calF_i^{(n+d)}=\emptyset$
for all $0 \leq i<n$, hence $X= \calF_n^{(n+d)}$.
		Condition (1) then
tells us that for every $\frakp\in X$, there
are $t_1,\dots,t_n\in\{1,\dots,{n+d}\}$ such
that $(a_{t_1}(\frakp),\dots,a_{t_n}(\frakp))\in V_{n,n}(R(\frakp))=
		V_n(R(\frakp))$. In particular,
$a_1(\frakp),\dots,a_{n+d}(\frakp)$ generate $A(\frakp)$ as
an $R(\frakp)$-algebra
for every $\frakp\in X$.  Lemma~\ref{LM:local-generation}
now implies that $a_1, \dots, a_{n+d}$ generate $A$ as an $R$-algebra,
proving the theorem.
		
To prove the claim, we will construct the elements $a_1,\dots,a_j \in A$
and the partition
$X=\calF_0^{(j)}\sqcup \calF_1^{(j)}\sqcup \dots\sqcup
\calF_n^{(j)}$
by induction on $j$.  For the base case $j = 0$,
set $\calF^{(0)}_0 := X$ and
$\calF^{(0)}_1 =\dots =\calF^{(0)}_n : =\emptyset$.
Condition (2)  clearly holds and
condition (1) is vacuous.
		
For the induction step,
assume that  elements $a_1,\dots,a_j \in A$
and a partition $X=\calF_0^{(j)}\sqcup \calF_1^{(j)}\sqcup \dots\sqcup
\calF_n^{(j)}$ satisfying conditions (1) and (2)
have been constructed for some $0 \leqslant j < n + d$.
We shall choose an  element
$a_{j+1} \in A$ as follows.
For each $0 \leqslant i < n$,  choose finitely many distinct
points $\frakp_{i,1}, \dots, \frakp_{i, N_i} \in \calF_i^{(j)}$
meeting all irreducible components of $\calF_i^{(j)}$ (here we
are using our standing assumption that $R$ is Noetherian).
By condition (1), for each point
$\frakp_{i, s}$, there exist integers
$t_1,\dots,t_i\in\{1,\dots,j\}$ (depending on $i$ and $s$)
such that $(a_{t_1}(\frakp_{i, s}),\dots,a_{t_i}(\frakp_{i, s}))\in V_{n,i}
(R(\frakp_{i, s}))$.  Therefore, for each point $\frakp_{i, s}$
		there exists  $b_{i,s} \in A(\frakp_{i,s})$
		such that $(a_{t_1}(\frakp_{i,s}),\dots,a_{t_i}(\frakp_{i,s}),
b_{i,s}) \in V_{n,i+1}(R(\frakp_{i, s}))$.
Since the sets $\calF_0^{(j)}, \calF_1^{(j)}, \dots, \calF_{n-1}^{(j)}$
are disjoint, the points $\frakp_{i, s}$ are all distinct.
By the  Chinese Remainder Theorem, there exists
$a_{j+1}\in A$ such that $a_{j+1}(\frakp_{i, s})= b_{i,s}$
	for every $i = 1, \ldots, n-1$, and every $s = 1, \ldots, N_i$.
		
Now, by Lemma~\ref{LM:open-completable-generation}, for each $i$ and $s$
as above, there is an open subset $U_{i,s}$ of $X$ containing
$\frakp_{i, s}$ such that
		\begin{equation}\label{EQ:cond-ii}
(a_{t_1}(\frakp),\dots,a_{t_i}(\frakp),a_{j+1}(\frakp))\in V_{i+1,n}(R(\frakp))
		\end{equation}
		for all $\frakp \in U_{i, s}$. Let $U_i$ be the union of
$U_{i, s}$, as $s$ ranges from $1$ to $N_i$, and set $U_n=\emptyset$.
Now set
\begin{equation} \label{e.recursive}
 \calF_i^{(j+1)} : =
 \left\{ \begin{array}{ll}
(\calF_{i-1}^{(j)} \cap U_{i-1})  \cup
(\calF_i^{(j)} \setminus U_i) & \text{if $i = 1, \ldots, n$, and} \\
\calF_0^{(j)} \setminus U_0 & \text{if $i = 0$.}
\end{array}
\right.
\end{equation}
It is easy to see that $\{\calF_0^{(j+1)},\dots,\calF_n^{(j+1)}\}$ is a partition of $X$.
% (If $i = 0$, then $\calF_0^{(j+1)} : =  \calF_0^{(j)} \setminus U_0$.)
Let $0\leq i<n$. By our construction, $U_i$ meets all
irreducible components of $\calF_i^{(j)}$, hence
		\begin{equation}\label{EQ:cond-iii}
		\dim (\calF_i^{(j)} \setminus U_i) \leqslant
		\dim (\calF_i^{(j)}) - 1\ .
		\end{equation}
Conditions (1), (2) for the elements $a_1,\dots,a_j,a_{j+1}\in A$ and the partition $X=\bigsqcup_i\calF_i^{(j+1)}$
now readily  follow from \eqref{EQ:cond-ii}, \eqref{e.recursive}
and \eqref{EQ:cond-iii}.
This completes the proof of Theorem~\ref{thm.main}.
\qed
	
\section{Applications}
\label{sect.applications}

%  We  now apply Theorem~\ref{thm.main} to particular types of algebras.
% Throughout this section $F$ will denote a field, $R$ and $S$ will denote
% unital commutative associative rings, and $A$, $B$ denote finite $R$-algebras
% (not necessarily commutative, associative or unital).
% Our standing assumption
% is that these $R$-algebras are finitely generated as $R$-modules.
%
Let $A$ and $B$ be $R$-algebras.
We say that $A$ is a form of $B$ (or equivalently,
$B$ is a form of $A$) if there is a faithfully flat
commutative unital $R$-algebra $S$ such that $A\otimes_R S\cong B \otimes_R S$
as $S$-algebras.
% Many families of algebras can be defined as \emph{forms} of a specific
% algebra: If $S \in \Comm{R}$ and $B\in\Alg{S}$, we call
% $B$ a form of $A$ if there exists a faithfully flat $S'\in\Comm{S}$
% such that $A\otimes_RS'\cong B\otimes_SS'$ as $S'$-algebras.
For example,  an \emph{Azumaya} $R$-algebra of degree $n$
is a form of the matrix algebra $\Mat_n(R)$,
 a finite \emph{\'etale} $R$-algebra of rank $n$ is a form of
$R\times\dots\times R$ ($n$ times),
a \emph{Cayley} $R$-algebra is a form of the split octonion algebra
$\mathbb{O}_R$ (see~\cite[Corollary 4.11]{lpr}
or~\cite[Theorem 3.9]{petersson}), and  when
$2\in\units{R}$, an  \emph{Albert} $R$-algebra
is a form of the split Albert algebra
$H_3(\mathbb{O}_R)$, where $H_3$ denotes the space
of $3 \times 3$ Hermitian matrices
(see~\cite[Theorem~6.9]{petersson}).
%  \begin{itemize}
%  \item $A$ is an
%  an Azumaya $R$-algebra of degree $n$ if it is a form of the matrix
% algebra $\Mat_n(R)$.
% \item $A$ is an \'etale $R$-algebra of rank $n$ if it is a form of
%   $R \times\dots\times R$ ($n$ times).
% \item $A$ is a Cayley $R$-algebra if it is a form of
% the split octonion algebra
% $\mathbb{O}_R$; see~\cite[Corollary 4.11]{lpr}
% or~\cite[Theorem 3.9]{petersson}.
% \item When $2\in\units{R}$, $A$ is an
% Albert $R$-algebra if it is a form of the split Albert algebra
% $H_3(\mathbb{O}_R)$, where $H_3$ denotes the space
% of $3 \times 3$ Hermitian matrices;
% see~\cite[Theorem 6.9]{petersson}.\thefootnote{
% If $2$ is not invertible in $R$, then an Albert $R$-algebra should
% be regarded as \emph{quadratic} Jordan algebra~\cite[Section 4]{petersson};
% cf.\ Remark (iii) in the Introduction.}
%  \end{itemize}

\begin{prp}\label{PR:generators-of-forms}
Let $A$ and $B$ be finite-dimensional algebras over an infinite field $F$
and
assume $A$ is a form of $B$. If $B$ can be generated by $n$ elements,
then $A$ can also be generated by $n$ elements.
\end{prp}

\begin{proof}
Let $r := \dim_F(A) = \dim_F(B)$. Choose an $F$-basis for $A$ and
use it to identify $A$ with the $F$-points of the affine space $\bbA_F^r$.
It is easy to see that
there exists an open subscheme $U$ of  $(\bbA^r_F)^n$ such that for every
field extension $K/F$, the $K$-points of $U$ are the $n$-tuples
$(x_1, \dots, x_n) \in A_K^n$  that generate
$A_K$ as a $K$-algebra. Our goal is to show that $U$ has an $F$-point. Since $U$ is an open subscheme
of an affine space and $F$ is an infinite field, it suffices to check that
$U \neq \emptyset$.

Choose an $F$-field $K$  such that $A_K \cong B_K$. Since $B$ is generated by $n$ elements
as an $F$-algebra, the same $n$ elements will
generate $B_K$ as a $K$-algebra. As $A_K \cong B_K$, this implies that
$U$ has a $K$-point. Hence $U \neq \emptyset$, as claimed.
\end{proof}

\begin{cor}\label{cor.main} %Let $R$ be a commutative ring with $1$.
Assume the dimension of $\Max R$ is $d$. Then
\begin{enumerate}
\item[(a)] every Azumaya $R$-algebra  is generated by $d + 2$ elements,
\item[(b)] every Cayley $R$-algebra is generated by $d+3$ elements,
\item[(c)] every Albert $R$-algebra is generated by $d+3$ elements,
provided $2\in\units{R}$, \footnote{If $2$ is not
invertible in $R$, then an Albert $R$-algebra should
be regarded as \emph{quadratic} Jordan algebra~\cite[Section 4]{petersson};
cf.\ Remark (iii) in the Introduction.}
and
% \end{enumerate}
% When $R/\frakp$ is infinite for all $\frakp\in\Max R$, we further have
% \begin{enumerate}
\item[(d)] every finite \'etale $R$-algebra of rank $n$ is generated by $d+1$ elements,
provided $R/\frakp$ is infinite for any $\frakp\in\Max R$.
\end{enumerate}
\end{cor}

We remind the reader that in the statement of the corollary, ``generated''
means ``generated as a non-unital algebra''; allowing  use of the unit element
in the proof does not improve the bounds.
Note that, in (d),  the assumption that $R/\frakp$ is infinite  is automatic
if $R$ contains an infinite field.
%Similarly the condition on $R$ in part (d) is satisfied whenever
%$R$ contains an infinite field of characteristic $\neq 2$.

 \begin{proof}
 By Theorem~\ref{thm.main}, it is enough to prove the corollary
 when $R$ is a field $F$, in which case $d=0$.
 We let $A$ denote an $F$-algebra that is Azumaya (resp.\ Cayley, Albert, \'etale of rank $n$).

(a) %Let $\frakp\in\Max R$ and let $F := R/\frakp$.
%Then $B:=A(\frakp)$ is a central simple (associative)
%$F$-algebra.
%Azumaya $F$-algebras are precisely the central simple $F$-algebras.
%We need to show that any By Theorem~\ref{thm.main}, it is enough to show that $B$ is generated by $2$ elements over $F$.
First note that $\nMat{F}{n}$ is generated by the two matrices,
$E_{1,1}$ and $E_{1,2} + \dots + E_{n-1, n} + E_{n, 1}$. Here $E_{i,j}$
denotes the $n \times n$ matrix having $1$ in the
$(i, j)$-position and $0$ elsewhere.
Proposition~\ref{PR:generators-of-forms} now tells us
that when $F$ is infinite, any form of $\nMat{F}{n}$
is also generated by two elements.
When $F$ is a finite field, the only form
of $\nMat{F}{n}$ is $\nMat{F}{n}$ itself, by Wedderburn's theorem,
so we are done.

(b) By~\cite[\S III.4]{schafer}, a Cayley $F$-algebra $A$
is formed from a central simple $F$-algebra $Q$ of degree $2$ via
the Cayley--Dickson process. In particular, $A$ is generated by
one element over $Q$. As we saw in the proof of part (a),
$Q$ is generated by two elements over $F$. Hence, $A$ is generated
by three elements over $F$.

(c) A split Albert $F$-algebra
is generated by three elements; see \cite[p.~112]{mccrimmon}.
By Proposition~\ref{PR:generators-of-forms}, this is also the case for any Albert $F$-algebra when $F$ is infinite.
Thus we may assume that $F$ is finite. In this case
Serre's Conjecture I (proved by Steinberg) implies that
every Albert $F$-algebra is split.  Indeed, isomorphism classes
of Albert $F$-algebras are classified by the first Galois cohomology set
$\mathrm{H}^1(F, G)$, where $G$ is the split simply-connected
algebraic group of type ${\sf F_4}$ defined over
$F$~\cite[Proposition~37.11]{invbook}.
By Serre's Conjecture I, $\mathrm{H}^1(F, G)=0$ whenever $F$
has cohomological dimension $\leqslant 1$; see
\cite[Theorem~III.2.2.1]{Serre}. On the other hand,
finite fields are of cohomological dimension
$\leqslant 1$; see~\cite[Theorem~6.2.6, Proposition~6.2.3]{gs}.
This shows that $A$ is split, thus completing the proof of
part (c).

(d) We need to show that any \'etale $F$-algebra $A$ of rank $n$
over an infinite field $F$ is generated by a single element.
By Proposition~\ref{PR:generators-of-forms}, we may assume that
$A=F\times \dots\times F$ ($n$ times). In this case,
$A$ is generated by any element
$(\alpha_1, \dots, \alpha_n)$ with distinct entries.
\end{proof}

\begin{remark} \label{rem.etale}
In part (d), the assumption that $R/\frakp$ is infinite for any
$\frakp\in\Spec R$ cannot be removed in general.
Indeed,
when $R$ is a field $F$ with $q$ elements and $A=F\times \dots\times F$,  one needs at least $\lceil\log_q(n+1)\rceil$
generators,
since $x^q=x$ for any $x\in A$. In fact, it can be shown that any \'etale
$F$-algebra of rank $n$ can be generated by $\lceil\log_q(n+1)\rceil$ elements  (or $\lceil\log_qn\rceil$ if use of the unity is allowed).
Thus, if we drop the assumption that $R/\frakp$ is infinite
in Corollary~\ref{cor.main}(d),
we can still assert that $A$ is generated by $d+\lceil\log_q(n+1)\rceil$ elements, where $q=\min_{\frakp\in\Max R}|R/\frakp|$.
\end{remark}

\begin{remark}
Recall that a unital associative algebra
$A$ is called \emph{separable} if $A$ is projective
relative to the left $A\otimes_RA^\op$-module structure
given by $(a\otimes b^\op)x=axb$ ($a,b,x\in A$).
Examples of separable algebras include Azumaya
and finite \'etale algebras; see \cite{demeyer-ingraham} for further
details.
%Following Endo and Watanabe~\cite{endo-watanabe},
%$A$ is weakly semisimple if $A(\frakp)$ is a semisimple $R/\frakp$-algebra
%for every $\frakp \in \Max(R)$.
In the case where $\dim \, \Max(R) = d$ and
$R$ has no finite homomorphic images,
Corollary~\ref{cor.main}(a) can be generalized as follows:
{\em every finite separable $R$-algebra
can be generated by $d + 2$ elements.}
%
% \begin{cor} \label{cor.weakly-semisimple}
% Assume that $\dim \, \Max(R) = d$
% and $R/\frakp$ is  infinite for every $\frakp \in \Max(R)$.
% Then every weakly semisimple $R$-algebra
% can be generated by $d + 2$ elements.
% \end{cor}
%%%%%%%%%%%%%%%%%%%%%%%%%%%%%%%%%%%%%%%
% \begin{proof}
Indeed, by Theorem~\ref{thm.main} it suffices to show that
every separable algebra $B$ over an infinite field $F$
is generated by two elements. Let $K$ be an algebraic closure of $F$.
By  \cite[Corollary~2.4]{demeyer-ingraham}, $B\otimes_FK$ is a
product of matrix algebras $\Mat_{d_1}(K) \times \dots \times \Mat_{d_r}(K)$. By Proposition~\ref{PR:generators-of-forms}
we may assume $B$ itself is a product of matrix algebras
$\Mat_{d_1}(F) \times \dots \times \Mat_{d_r}(F)$.
In this case a proof can be found in~\cite[Proposition 2.10]{reichstein}.
\end{remark}

%%%%%%%%%%%%%%%%%%%%%%%
\noindent
{\bf Acknowledgement.}
We are grateful to Thomas R\"ud and the anonymous referee for their
help with the exposition.

\bibliographystyle{plain}

% \bibliographystyle{plain}
% \bibliography{../MyBib_15_09}

\end{document}